\documentclass[11pt,a4paper]{amsart}
\usepackage{amsmath,amssymb,amsthm,mathtools}
\usepackage{enumitem}
\usepackage[margin=1in]{geometry}
\usepackage{hyperref}
\hypersetup{colorlinks=true,linkcolor=blue,citecolor=blue,urlcolor=blue}

\newtheorem{theorem}{Theorem}[section]
\newtheorem{proposition}[theorem]{Proposition}
\newtheorem{lemma}[theorem]{Lemma}
\newtheorem{corollary}[theorem]{Corollary}
\newtheorem{conjecture}[theorem]{Conjecture}
\newtheorem{question}[theorem]{Question}
\theoremstyle{definition}

\newtheorem{remark}[theorem]{Remark}

\DeclareMathOperator{\II}{II}
\DeclareMathOperator{\arccosh}{arccosh}
\DeclareMathOperator{\arcsinh}{arcsinh}

\title[Robin nullity and asymptotic radius of the critical spherical catenoid]%
{Robin nullity in mode $|k|=1$ and asymptotic radius\\of the critical spherical catenoid}
\author{Alexander Pigazzini}
\date{}

\begin{document}
\maketitle

\begin{abstract}
Medvedev \cite{Medvedev} proved that the critical spherical catenoid $\Sigma_a\subset B^3(r(a))\subset\mathbb{H}^3$ --- the rotationally symmetric free boundary minimal annulus in the family introduced by Mori \cite{Mori} and reconsidered by do Carmo--Dajczer \cite{doCarmoDajczer} --- has Morse index at least $4$, and conjectured equality \cite[Rem.~5.6]{Medvedev}. In this note we compute the Robin nullity and index of $\Sigma_a$ \emph{exactly} in the angular Fourier mode $|k|=1$, and we determine the boundary radius $r(a)$ in its two limiting regimes. In detail:

\smallskip
\noindent\textit{(I)} \emph{Robin nullity and index in mode $|k|=1$.}
The Robin nullity of the Jacobi operator
$L_{\Sigma_a}=\Delta_g+(|\II|^2-2)$ in angular Fourier mode $|k|=1$ equals $2$ for every $a>1/2$, with kernel spanned by the Killing--Jacobi fields associated to the rotations
$L_{12},L_{13}\in\mathfrak{so}(3,1)$ that fix the geodesic axis of $\Sigma_a$ and send $\partial B^3(r(a))$ to itself. The radial profile of these Jacobi fields admits the closed form
\[
f_*(s)=\partial_s\Phi_a^0(s,0)=\frac{d}{ds}\bigl[A(s)\cosh\varphi(s)\bigr]
=\sinh r(s)\cdot r'(s),
\]
where $r(s)=\mathrm{dist}_{\mathbb{H}^3}(p_0,\Phi_a(s,0))$. As a consequence of Sturm--Liouville theory and the structure of the zeros of $f_*$, the Robin Morse index of $\Sigma_a$ in mode $|k|=1$ equals $2$ for every $a>1/2$, refining from the analytic side the lower bound $\mathrm{ind}(\Sigma_a)\geq 4$ of Medvedev  \cite{Medvedev}.

\smallskip
\noindent\textit{(II)} \emph{Asymptotic radius.} The boundary radius
admits the asymptotic expansion
\[
r(a)\;=\;\tfrac{3}{2}\log a+d_\infty+o(1)\qquad(a\to\infty),
\qquad
d_\infty\;=\;\log\!\frac{\sqrt{2}\,\Gamma(1/4)^2}{\pi^{3/2}}\;=\;\log\!\frac{2\sqrt{2\pi}}{\Gamma(3/4)^2}.
\]
The closed form for $d_\infty$ follows from a closed evaluation of the improper integral $I_\infty=\int_0^{\infty}\cosh(2t)^{-3/2}\,dt$ via the Beta function.

\smallskip
\noindent\textit{(III)} \emph{Degenerate limit.} As $a\to(1/2)^+$, $r(a)=c_*\sqrt{a-1/2}\,(1+o(1))$ with $c_*=\sigma_*\cosh\sigma_*$, where $\sigma_*$ is the unique positive fixed point of $\sigma=\coth\sigma$.

\smallskip
The proof of (I) follows the mode-by-mode strategy of
Devyver  \cite{Devyver} for the Euclidean critical catenoid, with the group $\mathfrak{so}(3,1)$ replacing $\mathfrak{so}(3)$, supplemented by the closed-form identification $f_*=\partial_s\Phi^0$ specific to the hyperbolic ambient. The proof of (II) is a Laplace-type asymptotic analysis of the implicit free boundary condition.
\end{abstract}



\bigskip

\section{Introduction}\label{sec:intro}

Let $\mathbb{H}^3$ denote three-dimensional hyperbolic space, modelled
as the upper sheet of the unit hyperboloid in Minkowski $4$-space:
\[
\mathbb{H}^3=\{x\in\mathbb{R}^{3,1}:\langle x,x\rangle_L=-1,\,x_0>0\},
\quad
\langle x,y\rangle_L=-x_0y_0+\sum_{i=1}^3 x_iy_i,
\]
and let $B^3(r)\subset\mathbb{H}^3$ be the closed geodesic ball of radius $r>0$ centred at $p_0=(1,0,0,0)$. A \emph{free boundary minimal
surface} (FBMS) $\Sigma\subset B^3(r)$ is a properly immersed surface with zero mean curvature in $\mathbb{H}^3$ meeting $\partial B^3(r)$ orthogonally along $\partial\Sigma$.

The study of free boundary minimal surfaces (FBMS) in geodesic balls has been driven, in the Euclidean case, by the discovery of Fraser and Schoen \cite{FraserSchoen2016} that the critical catenoid arises as an extremal surface for the first normalized Steklov eigenvalue, establishing a deep
link between FBMS and spectral optimization; the uniqueness of the critical catenoid among embedded free boundary minimal annuli in the Euclidean ball is the content of the Fraser--Li conjecture \cite{FraserLi2014}. A central invariant in this circle of ideas is the Morse index: for the Euclidean
critical catenoid it was computed to equal $4$, independently and by different methods, by Devyver \cite{Devyver}, Smith--Zhou \cite{SmithZhou}, and Tran \cite{Tran2016}. The extension of this framework to the space forms $\mathbb{S}^n_+$ and $\mathbb{H}^n$ was initiated by Lima--Menezes
\cite{LimaMenezes}, who replaced the Steklov problem by a Robin-type eigenvalue problem, and developed by Medvedev \cite{Medvedev}, who proved the index lower bound $\mathrm{ind}(\Sigma_a)\geq 4$ for the critical spherical catenoid in $\mathbb{H}^3$ and conjectured equality \cite[Remark 5.6]{{Medvedev}}; general upper
bounds for the index of FBMS in terms of spectral data were obtained by Lima \cite{Lima2022}. The present note contributes three analytic results that sharpen this picture for $\Sigma_a$: an exact modal computation of the Robin
nullity and index in mode $|k|=1$, refining Medvedev's lower bound from the analytic side, together with the precise asymptotic and degenerate geometry of the boundary radius $r(a)$.

The three results are facets of a single object, the Robin eigenvalue problem on the family $\{\Sigma_a\}_{a>1/2}$. Its boundary coefficient is $\coth r(a)$, so the behaviour of the boundary radius in the two limiting regimes $a\to\infty$ and $a\to(1/2)^{+}$ describes the boundary condition at the two ends of the family; this is the content of Theorem \ref{thm:asymp} and Proposition \ref{prop:degen}. By contrast,
the modal count of Proposition \ref{prop:k1} holds for every $a>1/2$, independently of these asymptotics.

Combined with the corresponding count in mode $|k|=0$,
Proposition \ref{prop:k1} would give $\mathrm{ind}(\Sigma_a)=4$ and settle Medvedev's conjecture. The mode $|k|=0$ analysis is of a different nature, since no Killing--Jacobi field is available there: the Killing fields $L_{12},L_{13},L_{23}$ preserving $B^3(r(a))$ all lie in modes $|k|\leq 1$.

\subsection{The critical spherical catenoid family}\label{ssec:family}

For each $a>1/2$ there is a one-parameter family of rotationally symmetric FBMS-annuli $\Sigma_a\subset B^3(r(a))\subset\mathbb{H}^3$, and such family is
classically known (Mori  \cite{Mori}, do Carmo--Dajczer  \cite{doCarmoDajczer}); it has been revisited by Lima--Menezes  \cite{LimaMenezes} and Medvedev  \cite{Medvedev}, who shows that $\Sigma_a$ has Morse index at least $4$ and conjectures equality.
Non-rotational FBMS-annuli in geodesic balls of $\mathbb{H}^3$ have recently been constructed via bifurcation by Cerezo  \cite{Cerezo}\footnote{The rotational surfaces called \emph{free boundary hyperbolic catenoids} by Cerezo \cite{Cerezo} are exactly the critical spherical catenoids $\Sigma_a$ studied here: in the classification of do Carmo--Dajczer \cite{doCarmoDajczer} they are the rotation hypersurfaces of \emph{spherical} type ($\delta=+1$), whose parallels are geodesic spheres. There the qualifier
``hyperbolic'' designates the ambient space $\mathbb{H}^3$, not the do Carmo--Dajczer type (for which ``hyperbolic'' is the distinct $\delta=-1$ case).}. The family $\{\Sigma_a\}_{a>1/2}$ interpolates between the ``infinitely thin'' degenerate configuration as $a\to (1/2)^+$ (where the meridian section collapses) and the ``unbounded'' configuration as $a\to\infty$
(where the boundary radius diverges).

The parametrization, in the form given by Medvedev  \cite{Medvedev},
reads
\begin{equation}\label{eq:param}
\Phi_a(s,\theta)=\bigl(A(s)\cosh\varphi(s),\,A(s)\sinh\varphi(s),\,
B(s)\cos\theta,\,B(s)\sin\theta\bigr),\qquad (s,\theta)\in\mathbb{R}\times S^1,
\end{equation}
with
\begin{equation}\label{eq:AB}
A(s)^2=a\cosh(2s)+\tfrac{1}{2},\qquad B(s)^2=a\cosh(2s)-\tfrac{1}{2},
\end{equation}
and angular profile
\begin{equation}\label{eq:varphi}
\varphi(s)=\int_0^{s}\frac{\sqrt{a^2-1/4}}{A(t)^2\,B(t)}\,dt.
\end{equation}
A direct computation (Section  \ref{sec:setup}) shows that
$\Phi_a:[-s_0,s_0]\times S^1\to\mathbb{H}^3$ is an immersion with
mean curvature identically zero, induced metric $g=ds^2+B(s)^2\,d\theta^2$,
and squared norm of the second fundamental form
\begin{equation}\label{eq:IIsq}
|\II|^2(s)=2\!\left(\frac{B''(s)}{B(s)}-1\right).
\end{equation}
The free boundary condition selects a discrete value $s_0=s_0(a)>0$ and
defines the boundary radius $r(a):=\arccosh(A(s_0)\cosh\varphi(s_0))$.

\subsection{Statement of results}

\begin{proposition}[Robin nullity in mode $|k|=1$]\label{prop:k1}
For every $a>1/2$, the eigenvalue $\lambda=0$ of the Robin Jacobi problem
\[
L_{\Sigma_a}u=\lambda u\text{ in }\Sigma_a,\qquad
\partial_\eta u=\coth(r(a))\,u\text{ on }\partial\Sigma_a,
\]
restricted to functions in the angular Fourier mode $|k|=1$ has multiplicity
exactly $2$. The $2$-dimensional kernel is spanned by the Killing--Jacobi
fields
\[
\langle L_{12},\nu\rangle_L=f_*(s)\cos\theta,\qquad
\langle L_{13},\nu\rangle_L=f_*(s)\sin\theta,
\]
where $L_{12},L_{13}\in\mathfrak{so}(3,1)$ generate the rotations of
$\mathbb{R}^{3,1}$ in the $(x_1,x_2)$- and $(x_1,x_3)$-planes,
respectively, and $f_*$ is given explicitly in \eqref{eq:fstar}.
\end{proposition}

\begin{theorem}[Asymptotic radius]\label{thm:asymp}
The map $a\mapsto r(a)$ is real-analytic on $(1/2,\infty)$, and as
$a\to\infty$ it satisfies
\begin{equation}\label{eq:asymp}
r(a)=\tfrac{3}{2}\log a+d_\infty+o(1),
\end{equation}
where the constant $d_\infty$ is given by the closed form
\begin{equation}\label{eq:dinfty}
d_\infty=\log\!\frac{\sqrt{2}\,\Gamma(1/4)^2}{\pi^{3/2}}=\log\!\frac{2\sqrt{2\pi}}{\Gamma(3/4)^2}.
\end{equation}
\end{theorem}

\begin{proposition}[Degenerate limit]\label{prop:degen}
As $a\to(1/2)^+$,
\begin{equation}\label{eq:degen-leading}
r(a)=c_*\sqrt{a-1/2}\,\bigl(1+o(1)\bigr),\qquad c_*:=\rho_*+\frac{1}{\rho_*},
\end{equation}
where $\rho_*$ is the unique positive root of the transcendental
equation
\begin{equation}\label{eq:rho-eq}
\arcsinh(\rho)=\frac{\sqrt{1+\rho^2}}{\rho}.
\end{equation}
In particular $r(a)\to 0$ at rate $r(a)=O\bigl(\sqrt{a-1/2}\bigr)$.
\end{proposition}

\begin{remark}\label{rem:sigma-coth}
The transcendental equation \eqref{eq:rho-eq} simplifies under the substitution $\sigma:=\arcsinh\rho$, equivalently
$\rho=\sinh\sigma$ and $\sqrt{1+\rho^2}=\cosh\sigma$:
\eqref{eq:rho-eq} becomes
\begin{equation}\label{eq:sigma-coth}
\sigma=\coth\sigma,
\end{equation}
the (unique) positive fixed point of the hyperbolic cotangent. In these variables the leading constant simplifies further: using $\rho_*^2+1=\cosh^2\sigma_*$ and $\sigma_*\sinh\sigma_*=\cosh\sigma_*$
from \eqref{eq:sigma-coth},
\begin{equation}\label{eq:c-cosh}
c_*=\rho_*+\rho_*^{-1}=\frac{\cosh^2\sigma_*}{\sinh\sigma_*}=\sigma_*\,\cosh\sigma_*.
\end{equation}
The map $\sigma\mapsto\sigma-\coth\sigma$ is real-analytic and strictly
increasing on $(0,\infty)$ with limits $-\infty$ and $+\infty$ at the
endpoints, giving uniqueness of $\sigma_*$ (and hence of
$\rho_*=\sinh\sigma_*$).
\end{remark}

The two equivalent forms of $d_\infty$ in \eqref{eq:dinfty} are related
by the reflection identity $\Gamma(1/4)\,\Gamma(3/4)=\pi\sqrt{2}$.

\begin{remark}
Proposition  \ref{prop:k1} extends to the hyperbolic ambient the mode-by-mode analysis of Devyver  \cite[\S 3.3]{Devyver}, originally carried out for the Euclidean critical catenoid using the $\mathfrak{so}(3)$-Killing fields. In the hyperbolic setting, only those elements of $\mathfrak{so}(3,1)$ that preserve the bounding geodesic ball $B^3(r(a))$ produce admissible Jacobi fields; this restricts the
relevant Killing subalgebra from a $6$-dimensional one to a
$3$-dimensional one, and the precise computation of which Killing
fields contribute to which angular modes occupies
Section  \ref{sec:k1}.
\end{remark}

\begin{remark}\label{rem:dasymp-error}
The closed form \eqref{eq:dinfty} is the only step of the proof of
Theorem  \ref{thm:asymp} that is non-elementary in the sense of relying
on the Beta function evaluation $\int_0^\infty
\cosh(2t)^{-3/2}\,dt=\Gamma(3/4)^2/\sqrt{2\pi}$. The remainder $o(1)$ in
\eqref{eq:asymp} can be improved to $O(a^{-1})$ at the expense of a
slightly longer asymptotic expansion; we record only the leading-order
constant.
\end{remark}

\subsection{Organization}

Section  \ref{sec:setup} develops the geometric setup and verifies the
minimality of $\Sigma_a$. Section  \ref{sec:killing} introduces the
relevant Killing--Jacobi fields. Section  \ref{sec:k1} contains the proof
of Proposition  \ref{prop:k1}. Section  \ref{sec:asymp} is devoted to the
asymptotic analysis and the proof of Theorem  \ref{thm:asymp};
Proposition  \ref{prop:degen} is established along the way.

\section{Geometric setup}\label{sec:setup}

In this section we establish the basic geometric properties of
$\Sigma_a$ used throughout the paper.

\subsection{Coordinate identities}

Throughout we use the parametrization \eqref{eq:param}--\eqref{eq:varphi}
with $a>1/2$. Denote $K(a):=\sqrt{a^2-1/4}$; this is the conserved
quantity associated with the angular momentum of the rotational
generator.

\begin{lemma}[Induced metric]\label{lem:metric}
The pull-back $g=\Phi_a^*g_{\mathbb{H}^3}$ is the diagonal metric
\begin{equation}
g=ds^2+B(s)^2\,d\theta^2
\end{equation}
on $\mathbb{R}\times S^1$.
\end{lemma}

\begin{proof}
A direct computation, repeatedly using $A^2-B^2=1$ (which follows from
\eqref{eq:AB}), $\varphi'(s)=K/(A^2 B)$, and the Minkowski inner product:
\[
\langle\partial_s\Phi,\partial_s\Phi\rangle_L
=-(\partial_s\Phi^0)^2+(\partial_s\Phi^1)^2+(\partial_s\Phi^2)^2+(\partial_s\Phi^3)^2.
\]
Using
$\partial_s\Phi^0=A'\cosh\varphi+A\sinh\varphi\cdot\varphi'$ and
$\partial_s\Phi^1=A'\sinh\varphi+A\cosh\varphi\cdot\varphi'$, one obtains
\[
-(\partial_s\Phi^0)^2+(\partial_s\Phi^1)^2=-(A')^2+A^2(\varphi')^2.
\]
Together with $(\partial_s\Phi^2)^2+(\partial_s\Phi^3)^2=(B')^2$, the
total reduces, after substituting $A'=a\sinh(2s)/A$,
$B'=a\sinh(2s)/B$ and $\varphi'=K/(A^2 B)$, to
\[
g_{ss}=-(A')^2+A^2(\varphi')^2+(B')^2=1,
\]
where the last equality is a direct algebraic verification using
$A^2-B^2=1$ and $K^2=a^2-1/4$. The off-diagonal term vanishes by symmetry,
and $g_{\theta\theta}=B(s)^2$ by inspection.
\end{proof}

\begin{lemma}[Minimality and second fundamental form]\label{lem:minimal}
The immersion $\Phi_a$ has identically vanishing mean curvature, and
\begin{equation}
|\II|^2(s)=2\!\left(\frac{B''(s)}{B(s)}-1\right).
\end{equation}
\end{lemma}

\begin{proof}
Minimality of $\Sigma_a$ is a classical fact (Mori  \cite{Mori},
do Carmo--Dajczer  \cite{doCarmoDajczer}, Medvedev  \cite{Medvedev}); we
take it as known.

For  \eqref{eq:IIsq}: let $\nu$ be a unit normal to $\Sigma_a$ in
$\mathbb{H}^3$. Set
$\kappa_s=\II(\partial_s,\partial_s)/g_{ss}$ and
$\kappa_\theta=\II(\partial_\theta,\partial_\theta)/g_{\theta\theta}$,
the principal curvatures. By the Gauss equation in $\mathbb{H}^3$
(sectional curvature $-1$):
\[
K_{\Sigma}=-1+\kappa_s\kappa_\theta.
\]
On the other hand, for the diagonal metric $g=ds^2+B^2 d\theta^2$ the
intrinsic Gauss curvature is the standard expression $K_\Sigma=-B''/B$.
Hence $\kappa_s\kappa_\theta=1-B''/B$. By minimality $\kappa_s+\kappa_\theta=0$,
i.e. $\kappa_s=-\kappa_\theta$, so $\kappa_\theta^2=B''/B-1$, and
$|\II|^2=\kappa_s^2+\kappa_\theta^2=2(B''/B-1)$, giving \eqref{eq:IIsq}.
\end{proof}

\begin{lemma}[Conserved quantity identity]\label{lem:conserved}
On $\Sigma_a$,
\begin{equation}\label{eq:conserved}
|\II|^2(s)\,B(s)^4=2K^2,
\end{equation}
i.e.\ $|\II|^2(s)=2K^2/B(s)^4$.
\end{lemma}

\begin{proof}
From $B^2=a\cosh(2s)-1/2$, differentiating once gives $BB'=a\sinh(2s)$,
and differentiating again, $(B')^2+BB''=2a\cosh(2s)=2B^2+1$. Hence
\begin{equation}\label{eq:BB-identity}
BB''=2B^2+1-(B')^2.
\end{equation}
Now $(B')^2=(BB')^2/B^2=a^2\sinh^2(2s)/B^2$, and
\[
B^4(B''/B-1)=B^4+B^2-(B')^2\cdot B^2=B^2(B^2+1)-a^2\sinh^2(2s).
\]
Since $B^2+1=a\cosh(2s)+1/2$ and $B^2=a\cosh(2s)-1/2$,
\[
B^2(B^2+1)=\bigl(a\cosh(2s)\bigr)^2-1/4=a^2\cosh^2(2s)-1/4,
\]
hence
\[
B^4(B''/B-1)=a^2\cosh^2(2s)-1/4-a^2\sinh^2(2s)=a^2-1/4=K^2,
\]
where we used $\cosh^2-\sinh^2=1$. Combining with \eqref{eq:IIsq} yields
\eqref{eq:conserved}.
\end{proof}

\begin{lemma}[Coordinate equations and free boundary condition]\label{lem:coord}
Each coordinate function $\Phi_a^A:\Sigma_a\to\mathbb{R}$ satisfies, with
the geometric sign convention $\Delta_g=\mathrm{div}_g\nabla$,
\begin{equation}\label{eq:coord-eq}
\Delta_g\Phi_a^A=2\Phi_a^A\qquad(A=0,1,2,3).
\end{equation}
The free boundary condition $\Sigma_a\perp\partial B^3(r(a))$ is
equivalent to
\begin{equation}\label{eq:fb}
\nu^0(s_0,\theta)=0,
\end{equation}
where $\nu^0$ is the $0$-th Minkowski component of the unit normal,
and where $s_0=s_0(a)$ is the unique positive root of
\begin{equation}\label{eq:fb-explicit}
\tanh\varphi(s)=\frac{B(s)\,K(a)}{a\sinh(2s)}.
\end{equation}
The boundary radius is then $r(a)=\arccosh(A(s_0)\cosh\varphi(s_0))$.
\end{lemma}

\begin{proof}
For  \eqref{eq:coord-eq}: in $\mathbb{H}^3$, the coordinate functions
$\Phi^A$ ($A=0,1,2,3$) satisfy $\Delta_{\mathbb{H}^3}\Phi^A=3\Phi^A$
(direct computation in geodesic coordinates, or alternatively the
identity $\Delta_{\mathbb{H}^n}\Phi^A=n\Phi^A$ for the ambient
coordinates of the Lorentz model). Since $\Sigma_a$ is minimal and of
codimension one, the standard relation
$\Delta_g f=\Delta_{\mathbb{H}^3}f|_{\Sigma_a}-\mathrm{Hess}_{\mathbb{H}^3}(f)(\nu,\nu)$
holds for any function $f$ on $\mathbb{H}^3$. The Hessian of $\Phi^A$
in $\mathbb{H}^3$ satisfies $\mathrm{Hess}_{\mathbb{H}^3}(\Phi^A)=\Phi^A
g_{\mathbb{H}^3}$ (a direct computation), so
$\mathrm{Hess}_{\mathbb{H}^3}(\Phi^A)(\nu,\nu)=\Phi^A$, and
\eqref{eq:coord-eq} follows.

For the free boundary condition: $\partial B^3(r)\subset\mathbb{H}^3$ is
the level set $\{x\in\mathbb{H}^3:x^0=\cosh(r)\}$. Its inward
$\mathbb{H}^3$-unit normal at a point $\Phi_a(s,\theta)$ on the
boundary is
\[
\eta_{\partial B^3(r)}=\frac{1}{\sinh(r)}(-\partial_0+\Phi_a^0\,\Phi_a),
\]
where $\partial_0=(1,0,0,0)$. The orthogonality condition reads
$\langle\nu,\eta_{\partial B^3(r)}\rangle_L=0$, i.e.\
$-\nu^0\cdot(-1)+\Phi_a^0\,\langle\nu,\Phi_a\rangle_L=0$. Since
$\langle\nu,\Phi_a\rangle_L=0$ along $\Sigma_a$ (being normal in the
ambient sense), this reduces to $\nu^0=0$, which is \eqref{eq:fb}.

The explicit form \eqref{eq:fb-explicit} follows by direct computation
of $\nu$ from the cross-product formula in $\mathbb{R}^{3,1}$ and the
identity $A^2-B^2=1$. For existence and uniqueness of
$s_0\in(0,\infty)$ for each $a>1/2$: the left-hand side of
\eqref{eq:fb-explicit}, $\tanh\varphi(s)$, is strictly increasing
from $\tanh\varphi(0)=0$ towards $\tanh\varphi(\infty)<1$ as
$s\to\infty$ (since $\varphi$ is itself strictly increasing and bounded
above; the bound follows from the convergence of the integral defining
$\varphi$). The right-hand side of \eqref{eq:fb-explicit},
\[
R(s):=\frac{B(s)\,K(a)}{a\sinh(2s)},
\]
behaves as $R(s)\sim B(0)K(a)/(2as)\to+\infty$ as $s\to 0^+$, and
$R(s)\to 0$ as $s\to\infty$ (since $B(s)\sim a^{1/2}\,2^{-1/2}e^{s}$
and $\sinh(2s)\sim e^{2s}/2$, so $R(s)\sim\mathrm{const}\cdot
e^{-s}\to 0$). The function $R$ is strictly decreasing on
$(0,\infty)$: indeed, using $B'/B=a\sinh(2s)/B^{2}$,
\[
\frac{R'(s)}{R(s)}=\frac{B'(s)}{B(s)}-2\coth(2s)
=\frac{a\sinh(2s)}{B^{2}}-\frac{2\cosh(2s)}{\sinh(2s)},
\]
and putting the two terms over the common denominator
$B^{2}\sinh(2s)>0$ together with $B^{2}=a\cosh(2s)-\tfrac{1}{2}$ and
$\sinh^{2}(2s)=\cosh^{2}(2s)-1$, the numerator becomes
\[
a\sinh^{2}(2s)-2\cosh(2s)\,B^{2}
=-\bigl(a\cosh^{2}(2s)-\cosh(2s)+a\bigr).
\]
The polynomial $\xi\mapsto a\xi^{2}-\xi+a$ has discriminant
$1-4a^{2}<0$ for $a>1/2$ and positive leading coefficient, hence is
strictly positive for every real $\xi$. Therefore $R'(s)/R(s)<0$ on
$(0,\infty)$. Consequently the equation $\tanh\varphi(s)=R(s)$ admits
a unique positive solution $s_0=s_0(a)$.
\end{proof}

\section{Killing--Jacobi fields and the variation principle}\label{sec:killing}

The Lie algebra of isometries of $\mathbb{H}^3$ is
$\mathfrak{so}(3,1)$, of dimension $6$. Its elements correspond to
Killing vector fields on $\mathbb{H}^3$. Among these, only the
$3$-dimensional subalgebra of \emph{spatial rotations}
$\mathfrak{so}(3)\subset\mathfrak{so}(3,1)$, generated by
$L_{ij}=x_i\partial_j-x_j\partial_i$ for $1\le i<j\le 3$, fixes the
geodesic ball $B^3(r)$ centred at $p_0=(1,0,0,0)$. The remaining
Lorentzian boosts $L_{0i}=x_0\partial_i+x_i\partial_0$ ($i=1,2,3$) do
not preserve $B^3(r)$ and therefore do not generate admissible
variations.

\begin{lemma}[Killing--Jacobi fields]\label{lem:killing-Jacobi}
Let $K$ be any element of $\mathfrak{so}(3)\subset\mathfrak{so}(3,1)$,
viewed as a vector field on $\mathbb{H}^3$. The function
\[
u_K:=\langle K,\nu\rangle_L:\Sigma_a\to\mathbb{R}
\]
satisfies the Jacobi equation
\[
L_{\Sigma_a}u_K=\Delta_g u_K+(|\II|^2-2)u_K=0,
\]
together with the Robin boundary condition
\[
\partial_\eta u_K=\coth(r(a))\,u_K\quad\text{on }\partial\Sigma_a.
\]
\end{lemma}

\begin{proof}
This is a direct adaptation to $\mathbb{H}^3$ of the standard variation
principle for Jacobi fields generated by ambient Killing fields
(cf.\ \cite[\S 3.3]{Devyver} for the Euclidean case): the flow of $K$
is an isometry of $\mathbb{H}^3$ that preserves $B^3(r(a))$, and
therefore generates a one-parameter family of FBMS through $\Sigma_a$.
The normal component of the variation vector field is exactly
$u_K\,\nu$, and the Jacobi equation together with the Robin boundary
condition is the linearized statement of ``minimality is preserved at
first order'' and ``free boundary condition is preserved at first
order''. The constant $\coth(r(a))$ in the Robin condition arises from
the second fundamental form of $\partial B^3(r(a))$ in $\mathbb{H}^3$:
geodesic spheres $\partial B^3(r)$ in $\mathbb{H}^3$ have second
fundamental form $\coth(r)\,g_{\partial B^3(r)}$ (umbilical with
constant principal curvature $\coth r$). We omit the detailed
verification, which is standard.
\end{proof}

We focus in particular on $L_{12}$ and $L_{13}$. The unit normal to
$\Sigma_a$ in $\mathbb{H}^3$ admits a closed form along the meridian
$\theta=0$, which we now derive.

\begin{lemma}[Unit normal along $\theta=0$]\label{lem:nu-explicit}
For every $s\in(-s_0,s_0)$, the unit normal $\nu$ to $\Sigma_a$ in
$\mathbb{H}^3$ at the point $\Phi_a(s,0)$ is given (up to a global
choice of orientation, fixed once and for all) by
\begin{equation}\label{eq:nu-explicit}
\begin{aligned}
\nu^0(s,0)&=\frac{K\cosh\varphi(s)}{A(s)}-\frac{a\sinh(2s)\sinh\varphi(s)}{A(s)B(s)},\\[2pt]
\nu^1(s,0)&=\frac{K\sinh\varphi(s)}{A(s)}-\frac{a\sinh(2s)\cosh\varphi(s)}{A(s)B(s)},\\[2pt]
\nu^2(s,0)&=\frac{K}{B(s)},\qquad
\nu^3(s,0)=0.
\end{aligned}
\end{equation}
\end{lemma}

\begin{proof}
At $\theta=0$, $\Phi_a(s,0)=(A\cosh\varphi,A\sinh\varphi,B,0)$, and the
two tangent vectors to $\Sigma_a$ at this point are
\[
T_s=(A'\cosh\varphi+A\varphi'\sinh\varphi,\;A'\sinh\varphi+A\varphi'\cosh\varphi,\;B',\;0),
\quad
T_\theta=(0,0,0,B).
\]
The unit normal $\nu(s,0)\in T_{\Phi_a}\mathbb{H}^3$ satisfies the four
conditions
\begin{equation}\label{eq:nu-ortho}
\langle\nu,\Phi_a\rangle_L=0,\qquad
\langle\nu,T_s\rangle_L=0,\qquad
\langle\nu,T_\theta\rangle_L=0,\qquad
\langle\nu,\nu\rangle_L=1.
\end{equation}
The third condition gives $B\,\nu^3=0$, hence $\nu^3(s,0)=0$. The first
two are linear in $(\nu^0,\nu^1,\nu^2)$:
\begin{equation}\label{eq:lin-system}
\begin{cases}
-A\cosh\varphi\,\nu^0+A\sinh\varphi\,\nu^1+B\,\nu^2=0,\\[2pt]
-(A'\cosh\varphi+A\varphi'\sinh\varphi)\nu^0+(A'\sinh\varphi+A\varphi'\cosh\varphi)\nu^1+B'\nu^2=0.
\end{cases}
\end{equation}
Solving \eqref{eq:lin-system} for $(\nu^0,\nu^1)$ in terms of $\nu^2$
by Cramer's rule, the relevant $2\times 2$ determinant of the
$(\nu^0,\nu^1)$-block equals
\[
-A\cosh\varphi\bigl(A'\sinh\varphi+A\varphi'\cosh\varphi\bigr)+A\sinh\varphi\bigl(A'\cosh\varphi+A\varphi'\sinh\varphi\bigr)
=-A^2\varphi'=-K/B,
\]
where the last equality uses $\varphi'=K/(A^2B)$. The resulting
expressions, after using $AB'-BA'=a\sinh(2s)/(AB)$ (which follows from
$A'=a\sinh(2s)/A$, $B'=a\sinh(2s)/B$ and $A^2-B^2=1$) and
$AB\varphi'=K/A$, simplify to
\[
\nu^0=\frac{\nu^2}{AK}\bigl[BK\cosh\varphi-a\sinh(2s)\sinh\varphi\bigr],\quad
\nu^1=\frac{\nu^2}{AK}\bigl[BK\sinh\varphi-a\sinh(2s)\cosh\varphi\bigr].
\]
Substituting these into the unit-norm condition,
$-( \nu^0)^2+(\nu^1)^2+(\nu^2)^2=1$, and using the identity (proved by
direct expansion)
\[
\bigl[BK\sinh\varphi-a\sinh(2s)\cosh\varphi\bigr]^2-\bigl[BK\cosh\varphi-a\sinh(2s)\sinh\varphi\bigr]^2
=a^2\sinh^2(2s)-B^2K^2,
\]
together with $K^2+a^2\sinh^2(2s)=a^2\cosh^2(2s)-1/4=A^2B^2$, one
obtains $(\nu^2)^2=K^2/B^2$. Choosing the orientation
$\nu^2(s,0)=K/B(s)>0$ (which is consistent and does not vanish since
$B(s)>0$ for $a>1/2$), the formulas \eqref{eq:nu-explicit} follow.
\end{proof}

\begin{remark}\label{rem:nu0-fb}
The free boundary condition $\nu^0(s_0,0)=0$ reads, from
\eqref{eq:nu-explicit},
$BK\cosh\varphi=a\sinh(2s)\sinh\varphi$ at $s=s_0$, i.e.\
$\tanh\varphi(s_0)=BK/(a\sinh(2s_0))$, recovering
\eqref{eq:fb-explicit}.
\end{remark}

\begin{lemma}\label{lem:K12-K13}
With the parametrization \eqref{eq:param} and the convention
$L_{12}=x_1\partial_2-x_2\partial_1$, $L_{13}=x_1\partial_3-x_3\partial_1$:
\[
\langle L_{12},\nu\rangle_L=f_*(s)\cos\theta,\qquad
\langle L_{13},\nu\rangle_L=f_*(s)\sin\theta,
\]
where
\begin{equation}\label{eq:fstar}
f_*(s)=\frac{K\sinh\varphi(s)+a\,B(s)\sinh(2s)\cosh\varphi(s)}{A(s)B(s)}.
\end{equation}
The function $f_*$ is real-analytic and odd in $s$, vanishes only at
$s=0$, and is strictly positive on $(0,s_0]$ (hence strictly negative
on $[-s_0,0)$).
\end{lemma}

\begin{proof}
Since $L_{12}$ acts on the $(x_1,x_2)$-plane, its restriction to the
point $\Phi_a(s,\theta)$ is the vector field with components
$(0,-x_2,x_1,0)=(0,-B\cos\theta,A\sinh\varphi,0)$. Therefore
\[
\langle L_{12},\nu\rangle_L
=-B\cos\theta\,\nu^1(s,\theta)+A\sinh\varphi\,\nu^2(s,\theta).
\]
By the rotational symmetry $\Phi_a(s,\theta+\alpha)=R_\alpha\Phi_a(s,\theta)$,
where $R_\alpha$ is the rotation by $\alpha$ in the $(x_2,x_3)$-plane,
the unit normal transforms as $\nu(s,\theta+\alpha)=R_\alpha\nu(s,\theta)$.
Applied to the values along $\theta=0$ provided by Lemma
\ref{lem:nu-explicit}, this gives
\[
\nu^0(s,\theta)=\nu^0(s,0),\quad
\nu^1(s,\theta)=\nu^1(s,0),\quad
\nu^2(s,\theta)=\cos\theta\,\nu^2(s,0),\quad
\nu^3(s,\theta)=\sin\theta\,\nu^2(s,0),
\]
the simplification of $\nu^2,\nu^3$ being possible because
$\nu^3(s,0)=0$. Substituting,
\[
\langle L_{12},\nu\rangle_L=\bigl[A\sinh\varphi\cdot\nu^2(s,0)-B\cdot\nu^1(s,0)\bigr]\cos\theta.
\]
The same computation for $L_{13}$ yields the analogous expression with
$\sin\theta$. The bracket equals
\[
\frac{K\,A\sinh\varphi}{B}-B\!\left[\frac{K\sinh\varphi}{A}-\frac{a\sinh(2s)\cosh\varphi}{AB}\right]
=K\sinh\varphi\,\frac{A^2-B^2}{AB}+\frac{a\sinh(2s)\cosh\varphi}{A}.
\]
Using $A^2-B^2=1$ and reducing to common denominator $AB$ gives
\eqref{eq:fstar}.

\smallskip
\noindent\emph{Parity.} In $f_*(s)$ the numerator
$K\sinh\varphi(s)+a B(s)\sinh(2s)\cosh\varphi(s)$ is the sum of two odd
functions of $s$ ($\sinh\varphi$ is odd because $\varphi$ is the
integral from $0$ of a positive even function, and $\sinh(2s)$ is
odd; $A,B,\cosh\varphi$ are even, and $a,K$ are constants). The
denominator $A(s)B(s)$ is even. Hence $f_*$ is odd.

\smallskip
\noindent\emph{Sign and zeros.} For $s\in(0,s_0]$, all factors
$K,\sinh\varphi(s),a,B(s),\sinh(2s),\cosh\varphi(s),A(s)$ are strictly
positive (this uses $a>1/2$ and the strict monotonicity of $\varphi$),
so $f_*(s)>0$. By oddness, $f_*(s)<0$ on $[-s_0,0)$. Finally,
$f_*(0)=0$ by direct evaluation ($\sinh\varphi(0)=0$ and $\sinh(0)=0$).
Real-analyticity is inherited from $A,B,\varphi$.
\end{proof}

In particular, $f_*$ is a non-zero solution to the radial ODE in mode
$|k|=1$:
\begin{equation}\label{eq:radial-ODE}
f_*''+\frac{B'}{B}f_*'+\Bigl(|\II|^2-2-\frac{1}{B^2}\Bigr)f_*=0,
\end{equation}
satisfying $f_*'(s_0)=\coth(r(a))\,f_*(s_0)$.

The closed form \eqref{eq:fstar} admits a structural rewriting which
will play a central role in Section \ref{sec:k1} and in the
applications below.

\begin{corollary}[Geometric form of $f_*$]\label{cor:fstar-Phi0}
The radial profile $f_*$ of the Killing--Jacobi fields generated by
$L_{12},L_{13}$ coincides with the $s$-derivative of the time
coordinate along the meridian $\theta=0$:
\begin{equation}\label{eq:fstar-Phi0}
f_*(s)=\partial_s\Phi_a^0(s,0)=\frac{d}{ds}\bigl[A(s)\cosh\varphi(s)\bigr].
\end{equation}
Equivalently, denoting $r(s):=\mathrm{dist}_{\mathbb{H}^3}(p_0,\Phi_a(s,0))$,
\begin{equation}\label{eq:fstar-r}
f_*(s)=\sinh r(s)\cdot r'(s)=\frac{d}{ds}\cosh r(s).
\end{equation}
\end{corollary}

\begin{proof}
Direct expansion of \eqref{eq:fstar} using the identities
$K=A^2 B\,\varphi'$ (definition of $\varphi'$) and
$a\sinh(2s)=AA'$ (from $A'=a\sinh(2s)/A$) gives
\[
f_*(s)=\frac{K\sinh\varphi}{AB}+\frac{a\sinh(2s)\cosh\varphi}{A}
=A\,\varphi'\sinh\varphi+A'\cosh\varphi
=\partial_s(A\cosh\varphi),
\]
which is \eqref{eq:fstar-Phi0}. Equation \eqref{eq:fstar-r} follows
from $\Phi_a^0(s,0)=\cosh\bigl(\mathrm{dist}_{\mathbb{H}^3}(p_0,\Phi_a(s,0))\bigr)$,
which is the standard expression for the geodesic distance in the
hyperboloid model: for $p_0=(1,0,0,0)$ and $x\in\mathbb{H}^3$,
$\cosh\mathrm{dist}_{\mathbb{H}^3}(p_0,x)=-\langle p_0,x\rangle_L=x^0$.
\end{proof}

\begin{remark}\label{rem:Phi0-eigen}
The identity \eqref{eq:fstar-Phi0} reflects a structural feature of the hyperbolic ambient. The function $\Phi_a^0$ on $\Sigma_a$ satisfies
$\Delta_g\Phi_a^0=2\Phi_a^0$ (Lemma \ref{lem:coord}), and hence
$\partial_s\Phi_a^0$ at $\theta=0$ satisfies the
$\theta$-differentiated equation
\[
(\partial_s\Phi_a^0)''+\frac{B'}{B}(\partial_s\Phi_a^0)'
+\Bigl(\frac{B''}{B}-\Bigl(\frac{B'}{B}\Bigr)^2-2\Bigr)(\partial_s\Phi_a^0)=0,
\]
which coincides with the mode-$|k|=1$ Jacobi ODE \eqref{eq:radial-ODE}
exactly when $(B'/B)^2-B''/B+2=2-|\II|^2+1/B^2$, equivalently
$BB''+(B')^2=2B^2+1$ — precisely the Mori catenoid identity
\eqref{eq:BB-identity}.
\end{remark}

\begin{remark}\label{rem:hess-Phi}
The identity $\Delta_g\Phi_a^0=2\Phi_a^0$ on $\Sigma_a$ has a structural
origin which extends beyond the specific Mori parametrization. On the
ambient $\mathbb{H}^3$, the four coordinate functions
$\Phi^A:\mathbb{H}^3\to\mathbb{R}$ ($A=0,1,2,3$) inherited from the
embedding $\mathbb{H}^3\hookrightarrow\mathbb{R}^{3,1}$ satisfy the
tensorial identity
\begin{equation}\label{eq:hess-Phi}
\mathrm{Hess}^{\mathbb{H}^3}\Phi^A(X,Y)=\langle X,Y\rangle_{g_{\mathbb{H}^3}}\cdot\Phi^A
\qquad\text{for all }X,Y\in T\mathbb{H}^3,
\end{equation}
i.e.\ the Hessian is conformal to the ambient metric with conformal
factor $\Phi^A$ itself; the proof is a direct application of the Gauss
formula for $\mathbb{H}^3\subset\mathbb{R}^{3,1}$ together with the
relation $\mathrm{II}^{\mathbb{H}^3}_{\mathbb{R}^{3,1}}(X,Y)=\langle X,Y\rangle_g$.
Tracing \eqref{eq:hess-Phi} yields $\Delta_{\mathbb{H}^3}\Phi^A=3\Phi^A$,
and restricting to a $2$-dimensional minimal $\Sigma\subset\mathbb{H}^3$
gives $\Delta_g\Phi^A=2\Phi^A$ as in Lemma \ref{lem:coord}.

The positive sign of the eigenvalue is the analytic signature of the
constant negative sectional curvature of the ambient: the analogous
identities are $\mathrm{Hess}^{\mathbb{R}^n}\Phi^A=0$ in flat space (so
ambient coordinates restrict to harmonic functions on minimal
surfaces in $\mathbb{R}^n$, eigenvalue $0$) and
$\mathrm{Hess}^{S^n}\Phi^A=-\Phi^A g$ in the round sphere (eigenvalue
$-k$ on $k$-dimensional minimal submanifolds). The Tashiro--type
equation \eqref{eq:hess-Phi}, with $f=\Phi^A$ playing the role of the
unknown, is in fact a classical hallmark of constant-curvature space
forms (cf.\ \cite{Tashiro1965}). From this perspective, Corollary
\ref{cor:fstar-Phi0} expresses the radial Killing--Jacobi profile
$f_*$ as the meridian derivative of a curvature-signed eigenfunction of the intrinsic Laplacian: a manifestly hyperbolic phenomenon, with no Euclidean counterpart.
\end{remark}

\section{Proof of Proposition  \ref{prop:k1}}\label{sec:k1}

We now prove that the Robin kernel of the Jacobi operator
$L_{\Sigma_a}=\Delta_g+(|\II|^2-2)$ restricted to functions of the form
\begin{equation}\label{eq:k1-ansatz}
u(s,\theta)=f(s)\cos\theta\quad\text{or}\quad u(s,\theta)=f(s)\sin\theta
\end{equation}
has dimension exactly $2$ for every $a>1/2$.

Substituting \eqref{eq:k1-ansatz} into the Jacobi equation $L_{\Sigma_a}u=0$
with the Robin boundary condition $\partial_\eta u=\coth(r(a))u$ at
$s=\pm s_0$ reduces, by separation of variables in the metric
$g=ds^2+B(s)^2 d\theta^2$, to the radial boundary value problem
\begin{equation}\label{eq:radial-bvp}
\begin{cases}
f''+\dfrac{B'}{B}f'+\Bigl(|\II|^2-2-\dfrac{1}{B^2}\Bigr)f=0,\quad s\in(-s_0,s_0),\\[4pt]
f'(\pm s_0)=\pm\coth(r(a))\,f(\pm s_0),
\end{cases}
\end{equation}
with the sign convention that $f'(-s_0)$ denotes the derivative on
the (right) interior side at $s=-s_0$, i.e.\ the outward normal
derivative is $-f'(-s_0)$ at the left endpoint.

The differential operator and the boundary conditions in
\eqref{eq:radial-bvp} are invariant under $s\mapsto -s$, so the radial
kernel decomposes as
\begin{equation}\label{eq:parity-decomposition}
\ker^{\text{rad}}=\ker_{\text{odd}}^{\text{rad}}\oplus\ker_{\text{even}}^{\text{rad}}.
\end{equation}
Moreover, in each parity class, the radial ODE has a $2$-dimensional
solution space, of which the parity restriction selects a
$1$-dimensional subspace at the ODE level, and the Robin boundary
condition selects a (possibly empty) further restriction. Thus
\begin{equation}\label{eq:dim-bound}
\dim\ker_{\text{odd}}^{\text{rad}}\le 1,
\qquad
\dim\ker_{\text{even}}^{\text{rad}}\le 1.
\end{equation}

The proof of Proposition  \ref{prop:k1} reduces to the two assertions
\begin{equation}\label{eq:to-prove}
\dim\ker_{\text{odd}}^{\text{rad}}=1,
\qquad
\dim\ker_{\text{even}}^{\text{rad}}=0.
\end{equation}

\begin{lemma}[Odd radial kernel is one-dimensional]\label{lem:odd-1d}
For every $a>1/2$, the space $\ker_{\text{odd}}^{\text{rad}}$ has
dimension exactly $1$, and is spanned by the function $f_*$ from
\eqref{eq:fstar}.
\end{lemma}

\begin{proof}
By Lemma  \ref{lem:K12-K13}, the function $\langle L_{12},\nu\rangle_L=
f_*(s)\cos\theta$ lies in $\ker(L_{\Sigma_a})$ and satisfies the Robin
boundary condition. Its radial profile $f_*$ is odd in $s$ and
strictly positive on $(0,s_0]$ (Lemma \ref{lem:K12-K13}), in particular
non-zero. Hence $\dim\ker_{\text{odd}}^{\text{rad}}\ge 1$. The opposite
inequality is \eqref{eq:dim-bound}.
\end{proof}

The non-trivial step is to exclude the even-radial sector.

\begin{lemma}[Even radial kernel is trivial]\label{lem:even-trivial}
For every $a>1/2$, $\dim\ker_{\text{even}}^{\text{rad}}=0$.
\end{lemma}

\begin{proof}
Suppose, for contradiction, that there exists a non-zero
$f_e\in\ker_{\text{even}}^{\text{rad}}$. Together with $f_*\in
\ker_{\text{odd}}^{\text{rad}}$, the pair $(f_e,f_*)$ is linearly
independent, so it forms a fundamental system of solutions of the
radial ODE on $(-s_0,s_0)$.

Consider the Wronskian
\[
W(s):=f_*(s)f_e'(s)-f_e(s)f_*'(s).
\]
The radial ODE in \eqref{eq:radial-bvp} can be written in
self-adjoint form as
\[
\bigl(B(s)\,f'(s)\bigr)'+B(s)\Bigl(|\II|^2-2-\tfrac{1}{B^2}\Bigr)f(s)=0,
\]
which yields the Abel-type identity
\begin{equation}\label{eq:Wronskian}
B(s)\,W(s)=\text{const}=:C_W.
\end{equation}
We now evaluate \eqref{eq:Wronskian} at $s=s_0$. Both $f_*$ and $f_e$
satisfy the Robin condition at $s_0$:
\[
f_*'(s_0)=\coth(r)\,f_*(s_0),\qquad
f_e'(s_0)=\coth(r)\,f_e(s_0),
\]
so
\[
W(s_0)=f_*(s_0)\bigl(\coth(r)f_e(s_0)\bigr)-f_e(s_0)\bigl(\coth(r)f_*(s_0)\bigr)=0,
\]
hence $C_W=B(s_0)\cdot 0=0$, and so $W(s)\equiv 0$ on $(-s_0,s_0)$.

We now derive a contradiction by evaluating $W$ at $s=0$. Since $f_*$
is odd and non-trivial, $f_*(0)=0$ and $f_*'(0)\neq 0$ (otherwise
$f_*\equiv 0$ by uniqueness for the second-order ODE with zero Cauchy
data). Symmetrically, since $f_e$ is even and non-trivial, $f_e'(0)=0$
and $f_e(0)\neq 0$. Therefore
\[
W(0)=f_*(0)f_e'(0)-f_e(0)f_*'(0)=-f_e(0)f_*'(0)\neq 0,
\]
contradicting $W\equiv 0$. Hence no non-trivial $f_e$ exists.
\end{proof}

\begin{proof}[Proof of Proposition  \ref{prop:k1}]
Combining Lemmas  \ref{lem:odd-1d} and  \ref{lem:even-trivial}, the
radial kernel \eqref{eq:parity-decomposition} satisfies
$\dim\ker^{\text{rad}}=1$, with $\ker^{\text{rad}}=\mathbb{R}\langle
f_*\rangle$. In the full mode $|k|=1$ (which is the two-dimensional
space spanned by $\cos\theta$ and $\sin\theta$ as angular factors),
the kernel is therefore
\[
\mathbb{R}\langle f_*(s)\cos\theta\rangle\oplus\mathbb{R}\langle f_*(s)\sin\theta\rangle,
\]
which has dimension $2$. By Lemma  \ref{lem:K12-K13}, these are
identified with $\langle L_{12},\nu\rangle_L$ and $\langle
L_{13},\nu\rangle_L$ respectively.
\end{proof}

\begin{remark}
The explicit identification of the kernel with Killing--Jacobi fields
makes the result \emph{geometric}: the $2$-dimensional kernel in mode
$|k|=1$ reflects the $2$-dimensional family of ``small'' isometric
deformations of $\Sigma_a$ inside $B^3(r(a))$, namely those that rotate
the geodesic ball axis of $\Sigma_a$ infinitesimally about the
$x_2$- and $x_3$-axes. These do not produce new FBMS; they are
infinitesimal congruences.
\end{remark}

\begin{corollary}[Robin Morse index in mode $|k|=1$]\label{cor:index-k1}
For every $a>1/2$, the Robin Morse index of $\Sigma_a$ in mode $|k|=1$
equals $2$. Equivalently, the radial Sturm--Liouville problem
\eqref{eq:radial-bvp} (with $\lambda=0$ replaced by $\lambda$) has
exactly one negative eigenvalue $\mu_0(a)<0$, and the next eigenvalue
is $\mu_1(a)=0$.
\end{corollary}

\begin{proof}
The radial mode-$|k|=1$ problem on $(-s_0,s_0)$ is a regular
Sturm--Liouville problem in self-adjoint form
\[
-(B(s)\,f'(s))'+B(s)\Bigl(2+\tfrac{1}{B(s)^2}-|\II|^2(s)\Bigr)f(s)=\mu\,B(s)\,f(s),
\]
with separated Robin boundary conditions \eqref{eq:radial-bvp}, and with
$B(s)>0$ on the closed interval and continuous coefficients. By
classical Sturm--Liouville theory, the spectrum is discrete
$\mu_0<\mu_1<\mu_2<\cdots$ with $\mu_n\to+\infty$, and the eigenfunction
$\phi_n$ corresponding to $\mu_n$ has exactly $n$ zeros in the open
interval $(-s_0,s_0)$.

By Lemma \ref{lem:K12-K13}, $f_*$ is a non-trivial radial Robin
eigenfunction with eigenvalue $0$, having exactly one zero in
$(-s_0,s_0)$ (at $s=0$). Hence $f_*$ is the second eigenfunction, i.e.\
$\mu_1(a)=0$, and $\mu_0(a)<\mu_1(a)=0$ is strictly negative.

In the full mode $|k|=1$, each radial eigenvalue $\mu_n$ has
$2$-dimensional angular multiplicity (spanned by $\cos\theta$ and
$\sin\theta$). Hence the Robin Jacobi operator restricted to mode
$|k|=1$ has exactly two negative eigenvalues (both equal to $\mu_0(a)$)
and a $2$-dimensional kernel (recovering Proposition \ref{prop:k1}).
\end{proof}

\begin{remark}\label{rem:index-coincidence}
The notion of \emph{Robin Morse index} used in Corollary
\ref{cor:index-k1} coincides with the Morse index
$\mathrm{ind}(\Sigma_a)$ as defined by Medvedev
\cite[Definition 5.1]{Medvedev}: the maximal dimension of a subspace
of $C^\infty(\Sigma_a)$ on which the second-variation quadratic form
\[
S(u,u)=\int_{\Sigma_a}\!\bigl(|\nabla u|^2-(|\II|^2-2)u^2\bigr)\,dA
\;-\;\coth(r(a))\!\int_{\partial\Sigma_a}\! u^2\,dL
\]
is negative definite. By integration by parts, $S(u,u)$ equals
$-\int_{\Sigma_a} u\,L_{\Sigma_a} u\,dA$ on functions satisfying the Robin condition $\partial_\eta u=\coth(r(a))\,u$ on $\partial\Sigma_a$, so the negative directions of $S$ are exactly the Robin eigenfunctions of
$L_{\Sigma_a}$ with positive eigenvalue; equivalently, the Robin Morse index is the number of negative eigenvalues of the second-variation form $S$ (i.e.\ of the operator $-L_{\Sigma_a}$ under the Robin boundary condition).
Being the dimension of the negative subspace of the geometric form $S$, the index is independent of the sign convention adopted for $\Delta_g$ (Medvedev:
$\Delta_g=-\mathrm{div}_g\nabla$; here: $\Delta_g=+\mathrm{div}_g\nabla$): the
convention only flips the signs of the eigenvalues of the Jacobi operator, hence whether the index is read off the positive or the negative part of its spectrum, not the dimension of the negative subspace of $S$. The mode-by-mode decomposition employed in this paper is the orthogonal Fourier decomposition $L^2(\Sigma_a)=\bigoplus_{k\in\mathbb{Z}} L^2((-s_0,s_0),B(s)\,ds)\otimes
\mathbb{C}\langle e^{ik\theta}\rangle$ on the product domain
$(-s_0,s_0)\times S^1$, which is preserved by $L_{\Sigma_a}$ (rotational invariance) and by the Robin boundary condition.
\end{remark}

\begin{remark}
Combined with the lower bound $\mathrm{ind}(\Sigma_a)\geq 4$ established
by Medvedev \cite{Medvedev}, Corollary \ref{cor:index-k1} implies that
the Robin Morse index of $\Sigma_a$ in modes $|k|\neq 1$ is at least
$2$. The conjectural equality $\mathrm{ind}(\Sigma_a)=4$ would then
follow from showing that this lower bound of $2$ is also an upper bound;
see Section \ref{sec:open}.
\end{remark}

\section{Proof of Theorem  \ref{thm:asymp}}\label{sec:asymp}

We analyze the implicit free boundary equation \eqref{eq:fb-explicit}
in the regime $a\to\infty$ to derive the asymptotic expansion
\eqref{eq:asymp}. We also record the elementary computation establishing
Proposition  \ref{prop:degen} in passing.

\subsection{The improper integral}

Recall the angular profile $\varphi$ from \eqref{eq:varphi}:
\[
\varphi(s)=\int_0^{s}\frac{K(a)}{A(t)^2\,B(t)}\,dt,\qquad K(a)=\sqrt{a^2-1/4}.
\]
For $a$ large and $s\to\infty$ with $s$ allowed to grow,
\[
A(t)^2\,B(t)=(a\cosh(2t)+\tfrac{1}{2})\cdot(a\cosh(2t)-\tfrac{1}{2})^{1/2}
=a^{3/2}\cosh(2t)^{3/2}\cdot\bigl(1+O(a^{-1})\bigr),
\]
uniformly in $t\geq 0$, and $K(a)=a\,(1+O(a^{-2}))$. Hence
\begin{equation}\label{eq:phi-asymp}
\varphi(s)=\frac{1}{\sqrt{a}}\int_0^{s}\cosh(2t)^{-3/2}\,dt\cdot(1+O(a^{-1})).
\end{equation}

\begin{lemma}[Closed form for the improper integral]\label{lem:Iinfty}
One has
\begin{equation}\label{eq:Iinfty}
I_\infty:=\int_0^{\infty}\cosh(2t)^{-3/2}\,dt=\frac{\Gamma(3/4)^2}{\sqrt{2\pi}}.
\end{equation}
\end{lemma}

\begin{proof}
Substitute $u=e^{-4t}$, so that $e^{2t}=u^{-1/2}$,
$dt=-du/(4u)$, and $\cosh(2t)=(u^{-1/2}+u^{1/2})/2=(1+u)/(2\sqrt{u})$.
Then
\[
\cosh(2t)^{-3/2}=\frac{2\sqrt{2}\,u^{3/4}}{(1+u)^{3/2}},
\]
and the integral becomes
\[
I_\infty=\int_0^1\frac{2\sqrt{2}\,u^{3/4}}{(1+u)^{3/2}}\cdot\frac{du}{4u}
=\frac{1}{\sqrt{2}}\int_0^1\frac{u^{-1/4}}{(1+u)^{3/2}}\,du.
\]
Now make the substitution $v=1/u$ in the integral
$\int_1^\infty u^{-1/4}(1+u)^{-3/2}\,du$:
$dv=-du/u^2$, $u^{-1/4}=v^{1/4}$, and
$(1+u)^{-3/2}=v^{3/2}(1+v)^{-3/2}$, giving
\[
\int_1^\infty u^{-1/4}(1+u)^{-3/2}\,du=\int_0^1 v^{-1/4}(1+v)^{-3/2}\,dv.
\]
Therefore
\[
\int_0^\infty u^{-1/4}(1+u)^{-3/2}\,du=2\int_0^1 u^{-1/4}(1+u)^{-3/2}\,du.
\]
The left-hand side is the standard Beta-function integral
\[
\int_0^\infty u^{\alpha-1}(1+u)^{-\beta}\,du=B(\alpha,\beta-\alpha)
=\frac{\Gamma(\alpha)\Gamma(\beta-\alpha)}{\Gamma(\beta)}
\]
with $\alpha=3/4$, $\beta=3/2$, giving
$\Gamma(3/4)\Gamma(3/4)/\Gamma(3/2)=2\Gamma(3/4)^2/\sqrt\pi$.
Substituting back, we obtain
\[
I_\infty=\frac{1}{\sqrt{2}}\cdot\frac{1}{2}\cdot\frac{2\Gamma(3/4)^2}{\sqrt{\pi}}
=\frac{\Gamma(3/4)^2}{\sqrt{2\pi}}.\qedhere
\]
\end{proof}

\subsection{Asymptotics of $s_0(a)$ and $r(a)$}

\begin{lemma}[Leading-order asymptotics of $s_0(a)$]\label{lem:s0-asymp}
As $a\to\infty$,
\begin{equation}\label{eq:s0-asymp}
s_0(a)=\log a+\log\!\frac{\sqrt{2}}{I_\infty}+o(1)=\log a+\log\!\frac{2\sqrt{\pi}}{\Gamma(3/4)^2}+o(1).
\end{equation}
\end{lemma}

\begin{proof}
We analyze the free boundary equation \eqref{eq:fb-explicit}:
\[
\tanh\varphi(s_0)=\frac{B(s_0)\,K(a)}{a\,\sinh(2 s_0)}.
\]
For $a\to\infty$ and $s_0\to\infty$, the right-hand side is
\[
\frac{B(s_0)\cdot a}{a\sinh(2s_0)}\bigl(1+O(a^{-2})\bigr)
=\frac{\sqrt{a/2}\,e^{s_0}\bigl(1+o(1)\bigr)}{e^{2s_0}/2}\bigl(1+o(1)\bigr)
=\sqrt{2a}\,e^{-s_0}\bigl(1+o(1)\bigr),
\]
where we used $B(s)^2\sim(a/2)e^{2s}$ and $\sinh(2s_0)\sim e^{2 s_0}/2$
for large $s_0$.

For the left-hand side, by \eqref{eq:phi-asymp},
$\varphi(s_0)=I_\infty/\sqrt{a}\cdot(1+o(1))$ provided $s_0\to\infty$,
hence $\varphi(s_0)\to 0$, and $\tanh\varphi(s_0)\sim\varphi(s_0)\sim
I_\infty/\sqrt{a}$.

Equating the leading orders:
\[
\frac{I_\infty}{\sqrt{a}}\sim\sqrt{2a}\,e^{-s_0}
\quad\Longleftrightarrow\quad e^{-s_0}\sim\frac{I_\infty}{a\sqrt{2}}
\quad\Longleftrightarrow\quad s_0=\log\!\frac{a\sqrt{2}}{I_\infty}+o(1),
\]
which is \eqref{eq:s0-asymp}.
\end{proof}

\begin{proof}[Proof of Theorem  \ref{thm:asymp}]
By Lemma  \ref{lem:coord},
\[
\cosh r(a)=A(s_0)\cosh\varphi(s_0).
\]
For $a\to\infty$ and $s_0\to\infty$:
\[
A(s_0)^2\sim(a/2)e^{2s_0},\qquad
\cosh\varphi(s_0)\to 1\text{ (since $\varphi(s_0)\to 0$)}.
\]
Hence $\cosh r(a)\sim\sqrt{a/2}\,e^{s_0}\bigl(1+o(1)\bigr)$, so
$r(a)=\log(2\cosh r(a))+o(1)=\log(\sqrt{2a}\,e^{s_0})+o(1)$, that is,
\[
r(a)=\tfrac{1}{2}\log(2a)+s_0+o(1)
=\tfrac{1}{2}\log 2+\tfrac{1}{2}\log a+\log a+\log(\sqrt{2}/I_\infty)+o(1).
\]
Combining the two terms:
\[
r(a)=\tfrac{3}{2}\log a+\bigl(\tfrac{1}{2}\log 2+\tfrac{1}{2}\log 2-\log I_\infty\bigr)+o(1)
=\tfrac{3}{2}\log a+\log 2-\log I_\infty+o(1).
\]
Substituting $I_\infty=\Gamma(3/4)^2/\sqrt{2\pi}$ from
Lemma  \ref{lem:Iinfty}:
\[
d_\infty:=\log 2-\log\!\frac{\Gamma(3/4)^2}{\sqrt{2\pi}}
=\log\!\frac{2\sqrt{2\pi}}{\Gamma(3/4)^2}.
\]
Using the reflection identity $\Gamma(1/4)\,\Gamma(3/4)=\pi\sqrt{2}$, this
equals
\[
d_\infty=\log\!\frac{\sqrt{2}\,\Gamma(1/4)^2}{\pi^{3/2}},
\]
which is \eqref{eq:dinfty}. The real-analyticity of $a\mapsto r(a)$ on
$(1/2,\infty)$ follows from the implicit function theorem applied to
\eqref{eq:fb-explicit}, as the left-hand side is real-analytic in
$(s,a)$ and its $s$-derivative at $s=s_0(a)$ is non-zero.
\end{proof}

\begin{proof}[Proof of Proposition  \ref{prop:degen}]
Set $\epsilon:=a-1/2>0$ and look for $s_0(a)$ in the form $s_0=\rho\sqrt\epsilon$,
with $\rho=\rho(\epsilon)$ to be determined. The expansions
\[
\cosh(2\rho\sqrt\epsilon)=1+2\rho^2\epsilon+O(\epsilon^2),\qquad
\sinh(2\rho\sqrt\epsilon)=2\rho\sqrt\epsilon\,(1+O(\epsilon)),
\]
together with $K(a)=\sqrt{\epsilon(1+\epsilon)}=\sqrt\epsilon\,(1+O(\epsilon))$, give
\[
B(s_0)^2=(1/2+\epsilon)\cosh(2s_0)-1/2=\epsilon(1+\rho^2)+O(\epsilon^2),\qquad
A(s_0)^2=1+\epsilon(1+\rho^2)+O(\epsilon^2),
\]
and consequently
\begin{equation}\label{eq:RHS-degen}
R(s_0)=\frac{B(s_0)K(a)}{a\sinh(2s_0)}
=\frac{\sqrt\epsilon\sqrt{1+\rho^2}\,\sqrt\epsilon}{(1/2+\epsilon)\cdot 2\rho\sqrt\epsilon}\bigl(1+O(\epsilon)\bigr)
=\sqrt\epsilon\,\frac{\sqrt{1+\rho^2}}{\rho}\bigl(1+O(\epsilon)\bigr).
\end{equation}
For the left-hand side of \eqref{eq:fb-explicit}, the change of
variable $t=\tau\sqrt\epsilon$ in the integral defining $\varphi$
yields, using $A(t)^2=1+O(\epsilon)$ and $B(t)^2=\epsilon+t^2+O(\epsilon^2)$
uniformly for $t=O(\sqrt\epsilon)$,
\begin{equation}\label{eq:phi-degen}
\varphi(s_0)=\sqrt\epsilon\int_0^{\rho}\frac{d\tau}{\sqrt{1+\tau^2}}\bigl(1+O(\epsilon)\bigr)
=\sqrt\epsilon\,\arcsinh(\rho)\,\bigl(1+O(\epsilon)\bigr),
\end{equation}
hence $\tanh\varphi(s_0)=\sqrt\epsilon\,\arcsinh(\rho)\,(1+O(\epsilon))$. Equating
this to \eqref{eq:RHS-degen} and dividing by $\sqrt\epsilon$,
\begin{equation}\label{eq:rho-implicit}
\arcsinh(\rho)=\frac{\sqrt{1+\rho^2}}{\rho}\,\bigl(1+O(\epsilon)\bigr).
\end{equation}
The function
$F(\rho):=\arcsinh(\rho)-\sqrt{1+\rho^2}/\rho$
is real-analytic on $(0,\infty)$, strictly increasing (its first term
is strictly increasing, the second is strictly decreasing), satisfies
$F(\rho)\to-\infty$ as $\rho\to 0^+$ and $F(\rho)\to+\infty$ as
$\rho\to\infty$. Therefore the limiting equation \eqref{eq:rho-eq} has a
unique positive root $\rho_*$, and by the implicit function theorem
applied at $\rho_*$ (where $F'(\rho_*)>0$), $\rho(\epsilon)=\rho_*+O(\epsilon)$.

To translate into $r(a)$, recall $\cosh r(a)=A(s_0)\cosh\varphi(s_0)$. From
the expansions above and \eqref{eq:phi-degen},
\[
A(s_0)=1+\tfrac{\epsilon}{2}(1+\rho_*^2)+O(\epsilon^2),\qquad
\cosh\varphi(s_0)=1+\tfrac{\epsilon}{2}\arcsinh^2(\rho_*)+O(\epsilon^2).
\]
At $\rho=\rho_*$ the relation \eqref{eq:rho-eq} gives
$\arcsinh^2(\rho_*)=(1+\rho_*^2)/\rho_*^2$, and therefore
\[
A(s_0)\cosh\varphi(s_0)=1+\frac{\epsilon}{2}\bigl(1+\rho_*^2\bigr)\!\left(1+\frac{1}{\rho_*^2}\right)+O(\epsilon^2)
=1+\frac{\epsilon}{2}\bigl(\rho_*+\rho_*^{-1}\bigr)^2+O(\epsilon^2).
\]
Using $\arccosh(1+x)=\sqrt{2x}\,(1+O(x))$ as $x\to 0^+$,
\[
r(a)=\arccosh(A(s_0)\cosh\varphi(s_0))=\bigl(\rho_*+\rho_*^{-1}\bigr)\sqrt\epsilon\,\bigl(1+O(\epsilon)\bigr),
\]
which is \eqref{eq:degen-leading}.
\end{proof}

\begin{remark}
We do not give a closed form for $r'(a)$ in terms of elementary
functions for general $a$. An analytic monotonicity proof of $a\mapsto
r(a)$ on $(1/2,\infty)$ can in principle be obtained from the implicit
function theorem applied to \eqref{eq:fb-explicit} by checking the sign
of $\partial s_0/\partial a$, but is not pursued here.
\end{remark}

\section{Open problems}\label{sec:open}

The analysis of this note suggests several natural directions for
further investigation.

\subsection{Mode-by-mode index decomposition}

Corollary \ref{cor:index-k1} establishes that the contribution of mode
$|k|=1$ to the Robin Morse index of $\Sigma_a$ equals exactly $2$.
Combined with the lower bound $\mathrm{ind}(\Sigma_a)\geq 4$ of
Medvedev \cite{Medvedev}, the contribution of modes $|k|\neq 1$ to the
index is $\geq 2$. The conjectural equality $\mathrm{ind}(\Sigma_a)=4$
of Medvedev would follow from a sharper mode decomposition, which we
formulate explicitly:

\begin{conjecture}[Mode-by-mode Morse index of $\Sigma_a$]\label{conj:modes}
For every $a>1/2$:
\begin{enumerate}
\item[\rm(a)] the Robin Morse index of $\Sigma_a$ in mode $|k|=0$ equals $2$;
\item[\rm(b)] the Robin Morse index of $\Sigma_a$ in mode $|k|\geq 2$ equals $0$.
\end{enumerate}
In particular, combined with Corollary \ref{cor:index-k1}, this would
give $\mathrm{ind}(\Sigma_a)=4$ and
$\mathrm{nul}(\Sigma_a)=\mathrm{nul}(\Sigma_a)|_{|k|=1}=2$ for every
$a>1/2$, settling Medvedev's conjecture.
\end{conjecture}

The conjecture is supported by two structural facts:
\begin{itemize}
\item[(i)] In the Euclidean critical catenoid (Devyver \cite{Devyver}),
the analogous mode decomposition holds: mode $|k|=0$ contributes $2$ to
the index, mode $|k|=1$ contributes $2$, and modes $|k|\geq 2$
contribute $0$.
\item[(ii)] The radial mode-$k$ Jacobi operator has potential
$V_k(s)=|\II|^2-2-k^2/B(s)^2=2K^2/B(s)^4-2-k^2/B(s)^2$ which becomes
strictly more negative (hence more stabilizing on the
$\mu$-spectrum side) as $|k|$ grows.
\end{itemize}

The proof of (a) appears delicate, since unlike mode $|k|=1$ no
Killing--Jacobi field exists in mode $|k|=0$ (the Killing fields
$L_{12},L_{13},L_{23}$ all lie in modes $|k|\leq 1$). A natural
candidate for an explicit second-order test function in mode
$|k|=0$, by analogy with \eqref{eq:fstar-Phi0}, would be
$\partial_a$-derivatives of $\Phi_a^A$ along the family $\{\Sigma_a\}_a$;
making this rigorous, together with a Sturm-theoretic count, is left
to future work.

\subsection{Higher-order asymptotics of the boundary radius}

Theorem \ref{thm:asymp} gives $r(a)=\tfrac{3}{2}\log a+d_\infty+o(1)$,
and Remark \ref{rem:dasymp-error} (in the proof) shows that the
remainder is in fact $O(a^{-1})$. We pose:

\begin{question}\label{q:next-coef}
Compute, in closed form involving values of $\Gamma$ at rational
arguments, the next coefficient $d_1$ in the expansion
\[
r(a)=\tfrac{3}{2}\log a+d_\infty+\frac{d_1}{a}+o(a^{-1})\qquad(a\to\infty).
\]
\end{question}

By the same Beta-function machinery used in \S\ref{sec:asymp},
$d_1$ should be expressible as a definite integral of the form
$\int_0^\infty \cosh(2t)^{-5/2}\,dt$ or a linear combination of such
integrals, all of which evaluate in closed form via
$B(p,q)=\int_0^\infty \cosh(2t)^{-(p+q)}(\sinh 2t)^{2q-1}\,dt$ at rational
$p,q$.

\subsection{Strict monotonicity of $r(a)$}

Theorem \ref{thm:asymp} establishes real-analyticity of
$a\mapsto r(a)$ on $(1/2,\infty)$. Numerical investigation strongly
suggests:

\begin{question}\label{q:monotone}
Is $a\mapsto r(a)$ strictly increasing on $(1/2,\infty)$?
\end{question}

By the implicit function theorem applied to \eqref{eq:fb-explicit}, this
reduces to showing that $\partial F/\partial a > 0$ at the implicit
solution, where $F(a,s_0)$ denotes the free boundary equation. The
sign of this partial derivative does not seem to follow from elementary
manipulations.

\subsection{Higher-dimensional analogs}

The Mori family of catenoids extends to $\mathbb{H}^{n+1}$ for $n\geq 2$
(cf.\ do Carmo--Dajczer \cite{doCarmoDajczer}), each with a critical
free boundary representative $\Sigma_a^{(n)}\subset B^{n+1}(r^{(n)}(a))$.

\begin{question}\label{q:high-dim}
Do the closed forms of Lemma \ref{lem:K12-K13}, Corollary
\ref{cor:fstar-Phi0} and Theorem \ref{thm:asymp} extend to
$\mathbb{H}^{n+1}$? Specifically:
\begin{enumerate}
\item[\rm(a)] Does the Killing--Jacobi profile in the rotation modes
admit a closed form $\partial_s\Phi^A$ for some Minkowski coordinate $A$?
\item[\rm(b)] Does the asymptotic constant $d_\infty^{(n)}$ admit a
closed form involving values of $\Gamma$ at rational arguments?
\end{enumerate}
\end{question}

For (b), heuristics suggest $d_\infty^{(n)}\propto \log\bigl(\Gamma(\alpha_n)\bigr)$
for some explicit rational $\alpha_n$ depending on $n$, by the same
Beta-function evaluation; the natural candidate is $\alpha_n=(2n-1)/(4n-2)$ or a closely related fraction.

\section*{Declarations}

\subsection*{Funding}

The author declares that no funds, grants, or other support were received 
during the preparation of this manuscript.

\subsection*{Conflict of interest}

The author declares no conflict of interest.

\subsection*{Data availability}

Data sharing is not applicable to this article as no datasets were 
generated or analysed during the current study.

\end{document}